\documentclass[11pt]{article}
\usepackage[latin9]{inputenc}
\usepackage{color}
\usepackage{amsmath}
\usepackage{amsthm}
\usepackage{amssymb}
\usepackage[unicode=true,pdfusetitle,
 bookmarks=true,bookmarksnumbered=false,bookmarksopen=false,
 breaklinks=false,pdfborder={0 0 1},backref=page,colorlinks=true]
 {hyperref}
\hypersetup{
 linkcolor=blue}

\makeatletter
\numberwithin{equation}{section}
\numberwithin{figure}{section}
\theoremstyle{plain}
\newtheorem{thm}{\protect\theoremname}[section]
\theoremstyle{plain}
\newtheorem{cor}[thm]{\protect\corollaryname}
\theoremstyle{plain}
\newtheorem{prop}[thm]{\protect\propositionname}
\theoremstyle{definition}
\newtheorem{defn}[thm]{\protect\definitionname}
\theoremstyle{definition}
\newtheorem{example}[thm]{\protect\examplename}
\theoremstyle{plain}
\newtheorem{lem}[thm]{\protect\lemmaname}

\@ifundefined{date}{}{\date{}}
\usepackage{graphicx}
\usepackage{appendix}

\usepackage{enumitem}
\usepackage[backref=page]{hyperref}

\makeatother

\providecommand{\corollaryname}{Corollary}
\providecommand{\definitionname}{Definition}
\providecommand{\examplename}{Example}
\providecommand{\lemmaname}{Lemma}
\providecommand{\propositionname}{Proposition}
\providecommand{\theoremname}{Theorem}

\begin{document}
\global\long\def\F{\mathrm{\mathbf{F}} }%
\global\long\def\Aut{\mathrm{Aut}}%
\global\long\def\C{\mathbf{C}}%
\global\long\def\H{\mathbb{H}}%
\global\long\def\U{\mathbf{U}}%
\global\long\def\P{\mathcal{P}}%
\global\long\def\ext{\mathrm{ext}}%
\global\long\def\hull{\mathrm{hull}}%
\global\long\def\triv{\mathrm{triv}}%
\global\long\def\Hom{\mathrm{Hom}}%

\global\long\def\trace{\mathrm{tr}}%
\global\long\def\End{\mathrm{End}}%

\global\long\def\L{\mathcal{L}}%
\global\long\def\W{\mathcal{W}}%
\global\long\def\E{\mathbb{E}}%
\global\long\def\SL{\mathrm{SL}}%
\global\long\def\R{\mathbf{R}}%
\global\long\def\Pairs{\mathrm{PowerPairs}}%
\global\long\def\Z{\mathbf{Z}}%
\global\long\def\rs{\to}%
\global\long\def\A{\mathcal{A}}%
\global\long\def\a{\mathbf{a}}%
\global\long\def\rsa{\rightsquigarrow}%
\global\long\def\D{\mathbf{D}}%
\global\long\def\b{\mathbf{b}}%
\global\long\def\df{\mathrm{def}}%
\global\long\def\eqdf{\stackrel{\df}{=}}%
\global\long\def\ZZ{\mathcal{Z}}%
\global\long\def\Tr{\mathrm{Tr}}%
\global\long\def\N{\mathbf{N}}%
\global\long\def\std{\mathrm{std}}%
\global\long\def\HS{\mathrm{H.S.}}%
\global\long\def\e{\mathbf{e}}%
\global\long\def\c{\mathbf{c}}%
\global\long\def\d{\mathbf{d}}%
\global\long\def\AA{\mathbf{A}}%
\global\long\def\BB{\mathbf{B}}%
\global\long\def\u{\mathbf{u}}%
\global\long\def\v{\mathbf{v}}%
\global\long\def\spec{\mathrm{spec}}%
\global\long\def\Ind{\mathrm{Ind}}%
\global\long\def\half{\frac{1}{2}}%
\global\long\def\Re{\mathrm{Re}}%
\global\long\def\Im{\mathrm{Im}}%
\global\long\def\Rect{\mathrm{Rect}}%
\global\long\def\Crit{\mathrm{Crit}}%
\global\long\def\Stab{\mathrm{Stab}}%
\global\long\def\SL{\mathrm{SL}}%
\global\long\def\TF{\mathsf{TF}}%
\global\long\def\p{\mathfrak{p}}%
\global\long\def\j{\mathbf{j}}%
\global\long\def\uB{\underline{B}}%
\global\long\def\tr{\mathrm{tr}}%
\global\long\def\rank{\mathrm{rank}}%
\global\long\def\K{\mathcal{K}}%
\global\long\def\hh{\mathbb{H}}%
\global\long\def\h{\mathfrak{h}}%

\global\long\def\EE{\mathcal{E}}%
\global\long\def\PSL{\mathrm{PSL}}%
\global\long\def\G{\mathcal{G}}%
\global\long\def\Int{\mathrm{Int}}%
\global\long\def\acc{\mathrm{acc}}%
\global\long\def\awl{\mathsf{awl}}%
\global\long\def\even{\mathrm{even}}%
\global\long\def\z{\mathbf{z}}%
\global\long\def\id{\mathrm{id}}%
\global\long\def\CC{\mathcal{C}}%
\global\long\def\cusp{\mathrm{cusp}}%
\global\long\def\new{\mathrm{new}}%

\global\long\def\LL{\mathbb{L}}%
\global\long\def\M{\mathbb{M}}%
\global\long\def\I{\mathcal{I}}%
\global\long\def\X{X}%
\global\long\def\free{\mathbf{F}}%
\global\long\def\into{\hookrightarrow}%
\global\long\def\Ext{\mathrm{Ext}}%

\title{Strongly convergent unitary representations\\
 of limit groups}
\author{Larsen Louder and Michael Magee\\
With appendix by Will Hide and Michael Magee}
\maketitle
\begin{abstract}
{\normalsize{}We prove that all finitely generated fully residually
free groups (limit groups) have a sequence of finite dimensional unitary
representations that `strongly converge' to the regular representation
of the group. The corresponding statement for finitely generated free
groups was proved by Haagerup and Thorbjørnsen in 2005. In fact, we
can take the unitary representations to arise from representations
of the group by permutation matrices, as was proved for free groups
by Bordenave and Collins.}{\normalsize\par}

{\normalsize{}As for Haagerup and Thorbjørnsen, the existence of such
representations implies that for any non-abelian limit group, the
Ext-invariant of the reduced $C^{*}$-algebra is not a group (has
non-invertible elements).}{\normalsize\par}

{\normalsize{}An important special case of our main theorem is in
application to the fundamental groups of closed orientable surfaces
of genus at least two. In this case, our results can be used as an
input to the methods previously developed by the authors of the appendix.
The output is a variation of our previous proof of Buser's 1984 conjecture
that there exist a sequence of closed hyperbolic surfaces with genera
tending to infinity and first eigenvalue of the Laplacian tending
to $\frac{1}{4}$. In this variation of the proof, the systoles of
the surfaces are bounded away from zero and the surfaces can be taken
to be arithmetic.}{\normalsize\par}

\newpage{}
\end{abstract}
{\small{}\tableofcontents{}}{\small\par}

\section{Introduction}

A discrete group $\Gamma$ is \emph{fully residually free }(FRF)\emph{
}if for any finite set $S\subset\Gamma$, there exists a homomorphism
$\Gamma\to\F$ that is injective on $S$ where $\F$ is a free group.
Finitely generated FRF groups are known to coincide with Sela's \emph{limit
groups }\cite{SelaI}\emph{, }so we use these two notions interchangeably
in the sequel.

For $N\in\N$ let $\U(N)$ denote the group of $N\times N$ complex
unitary matrices. For a discrete group $\Gamma$, $\lambda_{\Gamma}:\Gamma\to\End(\ell^{2}(\Gamma))$
is the left regular representation. It was an open problem for some
years, popularized by Voiculescu in \cite[Qu. 5.12]{voic_quasidiag},
whether for a finitely generated free group $\F$, there exists a
sequence of unitary representations $\{\rho_{i}:\F\to\U(N_{i})\}_{i=1}^{\infty}$
such that for any element $z\in\C[\F]$, 
\[
\limsup_{i\to\infty}\|\rho_{i}(z)\|\leq\|\lambda_{\F}(z)\|.
\]
The norm on the left is the operator norm on $\C^{N_{i}}$ with respect
to the standard Hermitian metric, and the norm on the right is the
operator norm on $\ell^{2}(\Gamma)$. This problem was solved in the
affirmative in a huge breakthrough by Haagerup and Thorbjørnsen \cite{HaagerupThr}.

In fact, following \cite{voic_quasidiag}, given that the reduced
$C^{*}$-algebra of $\F$ is simple by a result of Powers \cite{powers},
the inequality above can be improved automatically\footnote{See proof of Theorem \ref{thm:main-analytical-theorem} below for
details.} to
\begin{equation}
\lim_{i\to\infty}\|\rho_{i}(z)\|=\|\lambda_{\F}(z)\|\quad\forall z\in\C[\F].\label{eq:strong-convergence}
\end{equation}
This notion of convergence of a sequence of finite dimensional unitary
representations given by (\ref{eq:strong-convergence}) applies equally
as well to any discrete group $\Gamma$ and we refer to this as \emph{strong
convergence.}
\begin{thm}
\label{thm:main-analytical-theorem}Any limit group $\Gamma$ has
a sequence of finite dimensional unitary representations that strongly
converge to the regular representation of $\Gamma$. In fact, these
unitary representations can be taken to factor through
\begin{equation}
\Gamma\to S_{N}\xrightarrow{\mathrm{std}}\U(N-1)\label{eq:factoring}
\end{equation}
 for some varying $N$, where $S_{N}$ is the group of permutations
of $N$ letters, and $\mathrm{std}$ is the $N-1$ dimensional irreducible
component of the representation of $S_{N}$ by 0-1 matrices.
\end{thm}

It was proved by G. Baumslag in \cite{GBaumslag} that the fundamental
groups $\Lambda_{g}$ of closed orientable surfaces are FRF, and it
is also known \cite[pp. 414-415]{baumslag_residually_free} that the
fundamental groups of non-orientable surfaces $S$ with $\chi(S)\leq-2$
are FRF. This gives the following corollary of Theorem \ref{thm:main-analytical-theorem}.
\begin{cor}
\label{cor:surfaces}Let $\Gamma$ denote the fundamental group of
a connected closed surface $S$ that is either orientable with no
constraint on $\chi(S)$, or non-orientable with $\chi(S)\leq-2$.
Then $\Gamma$ has a sequence of finite dimensional unitary representations
that strongly converge to the regular representation. Moreover, they
can be taken to be of the form (\ref{eq:factoring}) for some varying
$N$.
\end{cor}

Corollary \ref{cor:surfaces} leaves open the cases of connected non-orientable
surfaces with $\chi=1$ ($\R P^{2}$), $\chi=0$ (the Klein bottle
$\R P^{2}\#\R P^{2}$), and $\chi=-1$ ($(\R P^{2})^{\#3}$). In all
these cases the corresponding fundamental groups are not FRF\footnote{The first two cases are easy to check, and the case of $\chi=-1$
is due to Lyndon \cite{LYNDONEQUATION}.}. The fundamental group of $\R P^{2}$ is $\Z/2\Z$, and its regular
representation is finite dimensional. We prove the following as an
addendum to our main result.
\begin{prop}
\label{prop:Klein-bottle}The fundamental groups\\
$\pi_{1}((\R P^{2})^{\#2})=\langle\,a,b\,|\,b^{-1}=aba^{-1}\,\rangle$
and $\pi_{1}((\R P^{2})^{\#3})=\langle\,a,b,c\,|\,a^{2}b^{2}c^{2}\,\rangle$
have sequences of finite dimensional unitary representations that
strongly converge to their respective regular representations.
\end{prop}

The proof of Theorem \ref{thm:main-analytical-theorem} revolves around
the following potential property of discrete groups that we introduce
here.
\begin{defn}
\label{def:c*RF}A discrete group $\Gamma$ is \emph{$C^{*}$-residually
free} if for any finite set $S$ and $\epsilon>0$, there is a homomorphism
$\phi:\Gamma\to\F$ with $\F$ free such that 
\[
\|\lambda_{\F}(\phi(z))\|\leq\|\lambda_{\Gamma}(z)\|+\epsilon.
\]
for all $z\in\C[\Gamma]$ supported on $S$ with unit $\ell^{1}$
norm.
\end{defn}

\begin{example}
\label{exa:amenable-extensions}Any extension $N\to G\xrightarrow{\phi}\F$
of a free group by an amenable group $N$ is $C^{*}$-residually free.
Indeed, since $N$ is amenable $1$ is weakly contained in the regular
representation of $N$. Then by Fell's continuity of induction (\cite{Fell},
\cite[Thm. F.3.5]{BdlHV}) we have that the quasi-regular representation
of $G$ on $\ell^{2}(G/N)$ is weakly contained in the regular representation
$G$, hence by \cite[Thm. F.4.4]{BdlHV} for any $z\in\C[\F]$
\[
\|\lambda_{G/N}(zN)\|=\|\lambda_{\F}(\phi(z))\|\leq\|\lambda_{G}(z)\|.
\]
\end{example}

Here we prove the following.
\begin{thm}
\label{thm:Limit-groups-are-C-free}Limit groups are $C^{*}$-residually
free.
\end{thm}

The converse to Theorem \ref{thm:Limit-groups-are-C-free} does not
hold: Example \ref{exa:amenable-extensions} shows that $\Z\times\F$
is $C^{*}$-residually free, but it is easy to see that it is not
FRF. It is, however, also easy to see that it is residually free.
Furthermore, the group 
\[
\langle a,b,c\thinspace|\thinspace b^{-1}=aba^{-1},\ [c,b]\rangle
\]
is $C^{*}$-residually free by Example \ref{exa:amenable-extensions}
since it is an extension of the free group $\langle b,c\rangle$ by
$\Z\cong\langle a\rangle$. On the other hand, it is not even residually
free since it contains an embedded Klein bottle subgroup $\langle a,b\rangle$.
It is an interesting question, not pursued here, to give some alternative
characterization of a group being $C^{*}$-residually free. 

Given a free group $\F$, and a basis $X$ of $\F$, we write $|f|_{X}$
for the word length of $f$ in the basis $X$. In any discrete group
$\Gamma$ with generating set we write $B_{Y}(r)$ for the elements
of $\Gamma$ that can be written a product of as most $r$ elements
of $Y\cup Y^{-1}$. The proof of Theorem \ref{thm:Limit-groups-are-C-free}
relies on the following key proposition. 
\begin{prop}
\label{prop:distortion}Let $\Gamma$ be a limit group with a fixed
finite generating set $Y$. There is $D=D(\Gamma,Y)>0$ and $C=C(\Gamma,Y)>0$
such that for any $r>0$ there is an epimorphism $f:\Gamma\to\F$
with $\F$ free, which is injective on $B_{Y}(r)$, and a basis $X$
of $\F$ such that
\[
\max_{g\in B_{Y}(r)}|f(g)|_{X}\leq Cr^{D}.
\]
\end{prop}

\subsection{Further consequences I: Spectral gaps}

A hyperbolic surface is a complete Riemannian surface (without boundary)
of constant curvature -1 . Given a hyperbolic surface $X$, we write
$\Delta_{X}$ for the Laplace-Beltrami operator on $L^{2}(X)$. If
$X$ is closed this operator's spectrum $\spec(\Delta_{X})$ consists
of eigenvalues $0=\lambda_{0}(X)\leq\lambda_{1}(X)\leq\cdots\leq\lambda_{k}(X)\leq\cdots$
with $\lambda_{k}(X)\to\infty$ as $k\to\infty$. It was a conjecture
of Buser \cite{Buser2} whether there exist a sequence of closed hyperbolic
surfaces $X_{i}$ with genera tending to infinity and with

\[
\lambda_{1}(X_{i})\to\frac{1}{4}
\]
where $\lambda_{1}$ denotes the first non-zero eigenvalue of the
Laplacian. The value $\frac{1}{4}$ is the asymptotically optimal
one by a result of Huber \cite{Huber}. See \cite[Introduction]{HM21}
for an overview of the rich history of this problem. Buser's conjecture
was settled in \cite{HM21}. The proof therein does not allow us to
take the surfaces to be arithmetic, and requires the surfaces to have
very short curves. The results of this work in conjunction with the
ideas in \cite{HM21} allow us, along with Hide, to prove: 
\begin{thm}
\label{thm:spectral-gap-surface}There exists a sequence of closed
arithmetic hyperbolic surfaces $\left\{ X_{i}\right\} _{i\in\mathbb{N}}$
with $g(X_{i})\to\infty$, systoles uniformly bounded away from zero,
and with 
\[
\lambda_{1}(X_{i})\to\frac{1}{4}.
\]
In fact the $X_{i}$ can be taken to be covering spaces of a fixed
arithmetic hyperbolic surface $X$. 
\end{thm}

Theorem \ref{thm:spectral-gap-surface} is proved in the Appendix\footnote{In fact, the Appendix proves a more general statement about coverings
of \emph{any }hyperbolic surface; see Theorem \ref{thm:Main Theorem}.} by the second named author (MM) and Hide, as a consequence of the
following corollary of Theorem \ref{thm:main-analytical-theorem}.
\begin{cor}[Matrix coefficients version of Theorem \ref{thm:main-analytical-theorem}]
\label{cor:matrix-coefficients}Let $\Gamma$ be a limit group. There
exist a sequence of finite dimensional representations $\rho_{i}$
such that for any $r\in\N$ and finitely supported map $a:\Gamma\to\mathrm{Mat}_{r\times r}(\C)$,
we have 
\[
\lim_{i\to\infty}\|\sum_{\gamma\in\Gamma}a(\gamma)\otimes\rho_{i}(\gamma)\|=\|\sum_{\gamma\in\Gamma}a(\gamma)\otimes\lambda(\gamma)\|.
\]
The norm on the left hand side is the operator norm for the tensor
product of ($r$ and $N_{i}$-dimensional) $\ell^{2}$ norms. The
norm on the right is the operator norm for the tensor product of $\ell^{2}$
and the inner product on $\ell^{2}(\Gamma)$. 

If $\Gamma$ is non-abelian then the $\rho_{i}$ can be taken of the
form (\ref{eq:factoring}).
\end{cor}

The proof of Corollary \ref{cor:matrix-coefficients} from Theorem
\ref{thm:main-analytical-theorem} is explained in $\S$\ref{sec:Proof-of-matrix-coefficients-version}.

\subsection{Further consequences II: $\protect\Ext(C_{r}^{*}(\Gamma))$ is not
a group}

In \cite{BrownDouglasFillmore1,BrownDouglasFillmore2}, Brown, Douglas,
and Fillmore introduced and studied a homological/K-theoretic invariant
$\Ext(\A)$ of a unital separable $C^{*}$-algebra $\A$. By definition,
$\Ext(\A)$ is the collection of $*$-homomorphisms 
\[
\pi:\A\to B(\ell^{2}(\N))/\K
\]
modulo conjugation of unitary operators on $\ell^{2}(\N)$, where
$B(\ell^{2}(\N))$ is the bounded operators on $\ell^{2}(\N)$ and
$\K$ is the ideal of compact operators therein. This is naturally
a semigroup with multiplication arising from $(\pi_{1},\pi_{2})\mapsto\pi_{1}\oplus\pi_{2}$
composed with an isomorphism $\ell^{2}(\N)\oplus\ell^{2}(\N)\cong\ell^{2}(\N)$. 

One of the motivations of the work of Haagerup and Thorbjørnsen \cite{HaagerupThr}
was to prove that there are non-invertible elements of $\Ext(C_{r}^{*}(\F))$
when $\F$ is a finitely generated non-abelian free group, i.e., $\Ext(C_{r}^{*}(\F))$
is not a group.

The passage from the existence of strongly convergent unitary representations
of $\F$ to this statement uses the following result proved by Voiculescu
in \cite[\S\S 5.14]{voic_quasidiag} (see \cite[Rmk. 8.6]{HaagerupThr}
for another exposition).
\begin{prop}
\label{prop:ext-condition-prop}If $\Gamma$ is a discrete, countable,
non-amenable group with a sequence of finite dimensional unitary representations
that strongly converge to the regular representation of $\Gamma$,
then $\Ext(C_{r}^{*}(\Gamma))$ is not a group.
\end{prop}

Since non-abelian limit groups $\Gamma$ are $C^{*}$-simple (Lemma
\ref{lem:simple}), they are non-amenable. Indeed, an amenable group
$\Gamma$ has a $C^{*}$-algebra morphism $C_{r}^{*}(\Gamma)\to\C$
by \cite[Thm. F.4.4]{BdlHV} whose kernel contradicts simplicity.
Hence combining Theorem \ref{thm:main-analytical-theorem} with Proposition
\ref{prop:ext-condition-prop} we obtain the following extension of
`$\Ext(C_{r}^{*}(\F))$ is not a group':
\begin{cor}
If $\Gamma$ is a non-abelian limit group, then $\Ext(C_{r}^{*}(\Gamma))$
is not a group.
\end{cor}

\subsection*{Acknowledgments}

We thank Beno\^{i}t Collins, Simon Marshall, and Dan Voiculescu for
comments around this work.

This project has received funding from the European Research Council
(ERC) under the European Union\textquoteright s Horizon 2020 research
and innovation programme (grant agreement No 949143).

\section{Background}

\subsubsection*{Groups }

We write $e$ for the identity in any group. For any group $\Gamma$,
$\C[\Gamma]$ denotes the group algebra of $\Gamma$ with complex
coefficients. For a free group $\F$ with a fixed set of generators
$\X,$ for each $h\in\F$, we write $|h|_{\X}$ for the reduced word
length of $h$ with respect to $X$. If $Y$ is a symmetric generating
set of any group $\Gamma$, we write $B_{Y}(r)\subset\Gamma$ for
the elements of $\Gamma$ that can be written as the product of at
most $r$ elements of $Y$.

\subsubsection*{Analysis }

Given a discrete group $\Gamma$, $\lambda_{\Gamma}:\Gamma\to\End(\ell^{2}(\Gamma))$
is the \emph{left regular representation} 
\[
\lambda_{\Gamma}(g)[f](h)\eqdf f(g^{-1}h).
\]
This representation extends by linearity to one of the convolution
algebra $\ell^{1}(\Gamma$). For $\psi\in\ell^{1}(\Gamma)$, since
$\lambda$ is unitary we have the basic inequality

\begin{equation}
\|\lambda(\psi)\|\leq\|\psi\|_{\ell^{1}}\label{eq:basic-inequality}
\end{equation}
where the norm on the left is operator norm. The \emph{reduced $C^{*}$-algebra}
of $\Gamma$, denoted $C_{r}^{*}(\Gamma)$, is the closure of $\lambda(\ell^{1}(\Gamma))$
with respect to the operator norm topology. A \emph{tracial state
}on a unital $C^{*}$ algebra $\mathcal{A}$ is a linear functional
$\tau$ such that $\tau(1)=1$, $\tau(a^{*}a)\geq0$ (in particular,
is real) for all $a\in\A$ , and $\tau(ab)=\tau(ba)$ for all $a,b\in\A$.

An important inequality due to Haagerup \cite{Haagerup} links the
operator norm in $\End(\ell^{2}(\F))$ and the $\ell^{2}$ norm in
$\C[\F]$.
\begin{lem}[Haagerup]
\label{lem:Haagerup-iunequality}Let $X$ denote a finite generating
set for a free group $\F$. Suppose that $a\in\C[\F]$ is supported
on $B_{X}(r)$. Then 
\[
\|\lambda_{\F}(a)\|\leq(r+1)^{\frac{3}{2}}\|a\|_{\ell^{2}}.
\]
\end{lem}

\begin{proof}
Haagerup in \cite[Lemma 1.4]{Haagerup} proved that 
\[
\|\lambda_{\F}(a)\|\leq\sum_{i=0}^{\infty}(i+1)\|a_{i}\|_{\ell^{2}}
\]
where $a_{i}$ is the function $a$ multiplied pointwise by the indicator
function of $B_{X}(i)\backslash B_{X}(i-1)$, i.e. the sphere of radius
$i$. If $a$ is supported on $B_{X}(r)$ then using Cauchy-Schwarz
above gives the result, since $\sum_{i=0}^{r}\|a_{i}\|_{\ell^{2}}^{2}=\|a\|_{\ell^{2}}^{2}$.
\end{proof}
There is also a more basic inequality in the reverse direction that
holds for arbitrary discrete groups. Suppose that $\Gamma$ is a discrete
group. Then let $\delta_{e}\in\ell^{2}(\Gamma)$ denote the indicator
function of the identity. We have for $a\in\C[\Gamma]$
\begin{equation}
\|a\|_{\ell^{2}}^{2}=\langle\lambda(a)\delta_{e},\lambda(a)\delta_{e}\rangle\leq\|\lambda_{\Gamma}(a)\|^{2}.\label{eq:reverse-inequality}
\end{equation}

\section{Proof of Proposition \ref{prop:distortion}}

The proof of Proposition \ref{prop:distortion} relies on the deep
fact that any limit group embeds in an iterated extension of centralizers
of a free group, and quantified versions of theorems of Gilbert and
Benjamin Baumslag.
\begin{defn}
Let $\Gamma$ be a limit group, $A<\Gamma$ a maximal abelian subgroup.
A group $\Gamma'=\Gamma*_{A}B$, $B=A\times\langle t\rangle$ is an
\emph{extension of centralizers} of $\Gamma$. A group $\Gamma$ is
an \emph{iterated extension of centralizers} if there is a chain of
subgroups 
\[
\free=\Gamma_{0}<\Gamma_{1}<\dotsb<\Gamma_{n}=\Gamma
\]
such that $\Gamma_{i+1}$ is an extension of centralizers of $\Gamma_{i}$.
The \emph{height} of the extension is $n$.
\end{defn}

Any iterated extension of centralizers is fully residually free, and
so are their finitely generated subgroups, hence such subgroups are
limit groups. Amazingly, the converse holds: any limit group actually
embeds in a (finitely) iterated extension of centralizers. This was
first claimed by Kharlampovich and Myasnikov in their papers on the
Tarski problem~\cite[Theorem 4]{kmii}. For a proof following Sela
see~\cite[Theorem~4.2]{cg}. The forward implication seems to be
contained in Lyndon's original paper on his free exponential group~\cite[last two paragraphs, page 533]{lyndon_parametric},
which is the direct limit over the family of all iterated extensions
of centralizers of $\free$, ordered by inclusion. See also~\cite[Theorem~C1]{bmr}.

Let $\Gamma$ be a limit group with some fixed generating set $Y$.
The \emph{distortion function of $\Gamma$ with respect to $Y$} is
the function 
\[
d_{Y}(r)=\min_{\substack{f\colon\Gamma\to\free\\
X\subset\F
}
}\max_{g\in B_{Y}(r)}\vert f(g)\vert_{X}\,,
\]
where the minimum is over all $f\colon\Gamma\to\free$ which are injective
on $B_{Y}(r)$ and $X$ which are bases of $\F$. The proof of Proposition
\ref{prop:distortion} is a recapitulation of the proof that an iterated
extension of centralizers is fully residually free in a way that lets
us bound the distortion function by a polynomial whose degree depends
on the height. We start with an improvement of Baumslag's power lemma.
\begin{lem}[{cf. \cite[Proposition 1]{GBaumslag}}]
\label{quantitative baumslag} Let $u,b_{1},\dotsc,b_{n}$, reduced
words in $\free$, with $u$ also cyclically reduced, nontrivial,
and not a proper power of another element. If 
\begin{equation}
w=\prod_{i=0}^{n}u^{k_{i}}b_{i}=e\label{eq:word}
\end{equation}
for 
\[
\min_{i>0}\{\vert k_{i}\vert\}>(8n+2)\cdot\max_{i\geq0}\{1,\vert b_{i}\vert/\vert u\vert\}\,,
\]
then $\left[u,b_{i}\right]=e$ for some $i$. 
\end{lem}

G. Baumslag proved the same thing if $w=e$ for infinitely many integral
values of each of the $k_{i}$. See also the proof of~\cite[Lemma~4.13]{wilton_exercises},
which has, implicitly, an effective version of Lemma~\ref{quantitative baumslag}
in it.
\begin{proof}
The proof is by induction on $n$. Clearly for $n=0$, if $u^{k_{0}}b_{0}=e$
then $b_{0}$ is a power of $u$ and hence commutes with $u$.

We begin by manipulating our hypothesis to a more convenient form
for the induction. If 
\[
\min_{i>0}\{\vert k_{i}\vert\}>(8n+2)\cdot\max_{i\geq0}\{1,\vert b_{i}\vert/\vert u\vert\}
\]
then 
\begin{equation}
|w|>\vert u\vert\cdot\sum_{i>0}\vert k_{i}\vert>(8n+2)\cdot\max_{i\geq0}\{\vert u\vert,\vert b_{i}\vert\}.\,\label{eq:modw-lower-bound}
\end{equation}
(Here $|w|$ is the non-reduced length of $w$.)

Let $T$ be the tree of cancellations for $w$. A vertex of $T$ is
\emph{special} if it either corresponds to an endpoint of one of the
subwords $u^{k_{i}}$, one of the $b_{i}$, or has valence at least
three. An embedded segment in $T$ with special endpoints and no special
vertices in its interior is a \emph{long edge}.

Every valence one vertex in $T$ is special, so there are at most
$2n+2$ of them. We now work out the maximal number of long edges
in a tree with at most $2n+2$ valence one vertices, which will happen
when the number $q_{\geq3}$ of vertices of valence at least 3 is
maximized. Let $q_{m}$ be the number of valence $m$ vertices in
$T$. Then 
\[
1=\chi(T)=\sum q_{m}(1-m/2)\leq\frac{2n+2}{2}-\frac{1}{2}q_{\geq3}
\]
implies $q_{\ge3}\leq2n$, there are at most $4n+2$ special vertices,
and there are at most $4n+1$ long edges.

The sum of the lengths of the long edges is $\vert w\vert/2$, so
there is a long edge of length at least $\vert w\vert/(8n+2)$, which
from (\ref{eq:modw-lower-bound}) is at least
\[
\max_{i\geq0}\{\vert u\vert,\vert b_{i}\vert+1\}\,.
\]
If this is the case, since the endpoints of the $b_{i}$ are special,
the long edge is covered only by subsegments of powers of $u$. Because
$u$ is not a proper power, the segment (with a fixed direction) corresponds
to a unique reduced expression of the form $u_{0}u^{a}u_{1}$ where
$u_{0}$ and $u_{1}$ are proper subwords of $u$ and $a>0$. (Otherwise,
one is led to the conclusion that $u$ can be written as a reduced
product of reduced words $u=pq=qp$, and by \cite[Lemma~2.2]{razborov},
this contradicts $u$ being a proper power.) Let us now fix the direction
of the long edge so $a>0$.

The upshot of this unique expression is that the term $u_{0}$ corresponds
to a terminal subsegment of a $u$ as written in (\ref{eq:word})
(part of a $u^{k_{i}}$ with $k_{i}>0$), for each time the long edge
is traversed in its given direction. If the long edge is traversed
in the other way by the path of $w$, then the $u_{0}$ segment corresponds
to an initial subsegment of a $u^{-1}$ in a $u^{k_{i}}$ with $k_{i}<0$.

Fix an endpoint $v$ of the $u_{0}$ segment in the long edge. Consider
the subpaths of the path of $w$ punctuated by returns to $v$. After
cutting the tree at $v$, there must be at least one $b_{i}$ subpath
on either half of the resulting forest. So there must be some closed
subpath of $w$ beginning and ending at $v$ and corresponding, possibly
after cyclic rotation of $w$, to a subsequence
\[
u^{k_{0}}b_{0}u^{k_{1}}b_{1}\dotsb u^{k_{j}-a}\underbrace{u^{a}b_{j}u^{k_{j+1}}\dotsb u^{k_{l}}b_{l}u^{c}}_{\downarrow}u^{k_{l+1}-c}b_{l+1}\dotsb u^{k_{n}}b_{n}
\]
with $l-j<n$, and
\[
u^{a}b_{j}u^{k_{j+1}}\dotsb u^{k_{l}}b_{l}u^{c}=e\,,
\]
which implies
\[
u^{a+c}b_{j}u^{k_{j+1}}\dotsb u^{k_{l}}b_{l}=e.
\]
Reducing $a+c$, we can use the inductive hypothesis to conclude that
for some $j$, $[u,b_{j}]=e$. (Note that this is where the minimum
of $k_{i}$ only over $i>0$ is useful in the induction; $a+c$ could
in principle be very small.)
\end{proof}
A similar result holds when $u$ is not necessarily cyclically reduced
and is a power, e.g., $u=ps^{l}p^{-1}$, with $s$ cyclically reduced
and $|u|=2|p|+l|s|$. Rewrite the expression for $w$ as 
\[
e=w=\prod ps^{lk_{i}}p^{-1}b_{i}
\]
conjugate by $p^{-1}$, and absorb the $p$'s into the $b$'s to get
\[
e=w'=\prod s^{lk_{i}}b_{i}'\,.
\]
Then the same conclusion clearly holds when 
\[
l\min_{i\geq0}\{\vert k_{i}|\}>(8n+2)\cdot\max_{i\geq0}\{1,(\vert b_{i}\vert+2\vert p\vert)/\vert s\vert\}\,.
\]
For the applications, since $|u|=l|s|+2|p|>2|p|$ so we can use instead
the easier to use yet still sufficient inequality
\begin{equation}
\min_{i\geq0}\{\vert k_{i}\vert\}\geq(8n+2)\cdot\max_{i\geq0}\{\vert b_{i}\vert+\vert u\vert\}\,,\label{the inequality}
\end{equation}
which gives the same conclusion. Note that this minimum of $k_{i}$
is now over all $i$, and not just $i>0$. (The latter was just more
convenient for the previous induction.)

In what follows, $\Gamma$ is a limit group with a fixed finite generating
set $Y$, $A$ is a maximal abelian subgroup in $\Gamma$, and $\Gamma'$
is the extension of centralizers $\Gamma'=\Gamma*_{A}B$, where $B=A\times\langle t\rangle$.
Let $Y'=Y\cup t$. The standard fact about amalgamated products lets
us write any element of $\Gamma'$ in normal form: 
\[
\prod_{i=0}^{n}\beta_{i}\gamma_{i}
\]
with $\gamma_{i}\in\Gamma$, $\beta_{i}\in B$, so that the only elements
which are allowed to be trivial are $\beta_{0}$ and $\gamma_{n}$,
and if any of them are in $A$ then the expression has length one
-- if, say, $\beta_{i}\in A$, then $\gamma_{i}\beta_{i}\gamma_{i+1}\in\Gamma$
and we group these together. Normal forms are unique (up to insertion
of $\alpha\alpha^{-1}$ pairs), but we will not use this fact. Given
an element as above, write each $\beta_{i}$ (uniquely) as $t^{n_{i}}\alpha_{i}$
with $\alpha_{i}\in A$ 
\[
\prod_{i=0}^{n}t^{n_{i}}\alpha_{i}\gamma_{i}
\]
and absorb the $\alpha_{i}$ into an adjacent $\gamma_{i}$ or $\gamma_{i+1}$
to get 
\[
\prod_{i=0}^{n}t^{n_{i}}v_{i}\:
\]
with each $v_{i}\in\Gamma$. If this isn't possible, leave it alone.
In this case the word is of the form 
\[
t^{n}\alpha\,,
\]
for some $\alpha\in A$. For the purposes of this argument, a word
is in normal form if it is of either of these two types.
\begin{lem}[QI lemma]
 \label{qi lemma} There is a constant $K$ such that if $w$ is
a word in $Y'$, then there is a word $w'=_{\Gamma}w$ in normal form
such that $\vert w'\vert\leq K\cdot\vert w\vert$. 
\end{lem}

\begin{proof}
The word $w$ can be put in normal form by replacing subwords which
are elements of $A$ in one step. Since $A$ is maximal abelian in
$\Gamma$ it is quasiconvex by \cite[Theorem~3.4]{alibegovic}, and
the rewritten word can only increase in length by a factor of $K$,
where $1/K$ is the shrinking factor of the embedding $A\into\Gamma'$. 
\end{proof}
\begin{lem}[{cf. \cite[Lemma~7, Theorem~8]{baumslag_residually_free}}]
 \label{inductive step} Let $K$ be the constant from Lemma \ref{qi lemma},
and fix $a\in A\backslash\{e\}$. Then 
\[
d_{Y'}(r/K)\leq(8r^{2}+4r)d_{Y}(2(r+\vert a\vert))^{2}\,.
\]
If $d_{Y}(r)$ is a polynomial of degree $D$ then $d_{Y'}(r)$ is
bounded above by a polynomial of degree $2D+2$. 
\end{lem}

This is essentially a version of B. Baumslag's generalization of G.
Baumslag's version of Lemma~\ref{quantitative baumslag} from free
groups to limit groups, where we keep track of the constants and avoid
the phrases ``sufficiently large'' and ``as large as we like.''
\begin{proof}
Let $\pi\colon\Gamma'\to\Gamma$ be the retraction to $\Gamma$ defined
by $\pi(t)=e$, let $\tau$ be the automorphism of $\Gamma'$ fixing
$\Gamma$ with 
\begin{align*}
\tau(t) & \eqdf ta.
\end{align*}
We will find an $h\colon\Gamma'\to\free$ which is injective on normal
forms in $Y'$ of length at most $r$ and doesn't stretch too much.

Suppose first that $g$ has the normal form
\[
\prod t^{n_{i}}v_{i}\,,
\]
which is a product of at most $\lfloor r/2\rfloor$ terms $t^{n_{i}}v_{i}$.
Then 
\[
f\circ\pi\circ\tau^{m}(g)=\prod f(a)^{mn_{i}}f(v_{i})\,.
\]
In order to use Lemma \ref{quantitative baumslag} we need to choose
$f$ so that $f(\left[a,v_{i}\right])\neq e$. In the worst case the
commutator $\left[a,v_{i}\right]$ has length at most 
\[
L\eqdf2(r+\vert a\vert).
\]
Choose $f\colon\Gamma\to\free$ and a basis $X$ of $\F$ such that
\[
d_{Y}(L)=\max_{g\in B_{Y}(L)}\vert f(g)\vert_{X}
\]
and $f$ embeds $B_{Y}(L)\subset\Gamma$. By~(\ref{the inequality}),
with $k_{i}=mn_{i}$, $n=\lfloor r/2\rfloor$, $u=f(a)$, and $b_{i}=f(v_{i})$,
as long as 
\[
\min\{\vert mn_{i}\vert\}\geq(4r+2)\cdot\max\{\vert f(v_{i})\vert_{X}+\vert f(a)\vert_{X}\}
\]
$f\circ\pi\circ\tau^{m}(g)$ is nontrivial. In the worst case $n_{i}=1$
for all $i$ and $f(a)$ and $f(v_{i})$ have length $d_{Y}(L)$,
so choose $m=m(r,\vert a\vert)=(4r+2)\cdot2d_{Y}(L)$ and let
\begin{align*}
h & =f\circ\pi\circ\tau^{m}.
\end{align*}
We continue to use the same basis $X$ for $\F$. Now overestimate
the length of $h(g)$: the normal form which can be expanded the most
is $t^{r}$, so we have $r\cdot m$ terms whose images have length
at most $d_{Y}(L)$, and therefore 
\[
\vert h(g)\vert_{X}\leq r\cdot\underbrace{(8r+4)\cdot d_{Y}(L)}_{m}\cdot\,d_{Y}(L)=(8r^{2}+4r)d_{Y}(L)^{2}\,.
\]

If $g$ is of the form $t^{n}\alpha$ then the worst that can happen
is $n=1$ and $\alpha$ has length at most $r-1$, $h(g)=f(a)^{m}f(\alpha)$,
but in this case 
\[
m\cdot d_{Y}(L)\geq m\cdot\vert f(a)\vert_{X}\geq\vert h(t)\vert_{X}\geq m>d_{Y}(r-1)\geq\vert f(\alpha)\vert_{X}\,,
\]
so $h(g)$ is nontrivial -- $h(t)$ and $f(\alpha)$ cannot fully
cancel since $m>d_{Y}(r-1)$, and is not longer than $m\cdot d_{Y}(L)+d_{Y}(r-1)$,
which is less than $r\cdot m\cdot d_{Y}(L)$.

Now by Lemma \ref{qi lemma} if $g\in B_{Y'}(r/K)$, it has a normal
form of length at most $r$ and $h(g)\neq e$, so $h$ embeds $B_{Y'}(r/K)$,
$\vert h(g)\vert_{X}\leq(8r^{2}+4r)d_{Y}(L)^{2}$, and the first part
of the lemma follows.

The statement about degrees is obvious since $L$ is linear in $r$. 
\end{proof}
\begin{cor}
Let $\Gamma$ be a limit group, and suppose $\Gamma$ embeds in an
extension of centralizers of height $n$. Then $d_{Y}(r)$ is bounded
above by a polynomial in $r$ of degree 
\[
D(n)=2^{n+2}-2^{n}-2\,.
\]
\end{cor}

\begin{proof}
For height $0$, the distortion function is just $r$. Clearly by
Lemma~\ref{inductive step} and induction a polynomial of degree
$D(n)$ suffices. Now embed $\Gamma$ in an iterated extension of
centralizers of height $n$: 
\[
\Gamma\into\Gamma_{n}>\Gamma_{n-1}>\dotsb>\Gamma_{1}>\free\,.
\]
Since the embedding $\Gamma\into\Gamma_{n}$ expands lengths at most
linearly, $\Gamma$ has distortion function bounded above by a polynomial
of degree $D(n)$ as well. 
\end{proof}
Proposition \ref{prop:distortion} follows immediately.

\section{Proof of Theorem \ref{thm:Limit-groups-are-C-free}\label{sec:Proof-of-Theorem-limit-group}}

\begin{proof}[Proof of Theorem \ref{thm:Limit-groups-are-C-free}]
Fix a set $Y$ of generators of $\Gamma$. It suffices to prove the
theorem for the finite set $B_{Y}(R)$ for arbitrary $R>0$. We are
given $\epsilon>0$. Let $S_{Y}(R)\subset\C[\Gamma]$ denote the $\ell^{1}$-unit
sphere of the elements supported on $B_{Y}(R)$. Our task is to prove
that there is a homomorphism $\phi:\Gamma\to\F$ with $\F$ free such
that  
\begin{equation}
\|\lambda_{\F}(\phi(a))\|\leq\|\lambda_{\Gamma}(a)\|+\epsilon.\label{eq:goal}
\end{equation}
for all $a\in S_{Y}(R)$. The set $S_{Y}(R)$ is compact with respect
to the $\ell^{1}$ norm. Take a finite $\frac{\epsilon}{3}$-net $\{a_{i}\}_{i\in\mathcal{I}}$
for $S_{Y}(R)$ w.r.t. the $\ell^{1}$ norm. 

Due to the inequality (\ref{eq:basic-inequality}) and triangle inequality,
the functions $a\mapsto\|\lambda_{\F}(a)\|$ and $a\mapsto\|\lambda_{\Gamma}(a)\|$
are 1-Lipschitz on $S_{Y}(R)$ with respect to the $\ell^{1}$ norm
and hence if we can prove the existence of $\phi:\Gamma\to\F$ with
$\F$ free such that
\begin{equation}
\|\lambda_{\F}(\phi(a_{i}))\|\leq\|\lambda_{\Gamma}(a_{i})\|+\frac{\epsilon}{3}\label{eq:on-net}
\end{equation}
for all $i\in\I$ then (\ref{eq:goal}) will follow for all $a\in S_{Y}(R)$
as required. So we now set out to prove (\ref{eq:on-net}).

Let $C$ and $D$ be the constants from Proposition \ref{prop:distortion}.
Choose $m=m(\epsilon)\in\N$ large enough so that
\begin{equation}
[C(2mR)^{D}]^{\frac{3}{4m}}\leq1+\frac{\epsilon}{3}.\label{eq:m-choice}
\end{equation}
We apply Proposition \ref{prop:distortion} with $r=2mR$ to get an
epimorphism $\phi:\Gamma\to\F$ injective on $B_{Y}(2mR)$, and a
generating set $X$ of $\F$ such that 
\begin{equation}
\phi(B_{Y}(2mR))\subset B_{X}\left(C(2mR)^{D}\right).\label{eq:supports}
\end{equation}
 Let $b_{i}\eqdf\phi(a_{i})$ for each $i\in\I$. 

Note that 
\[
\|\lambda_{\Gamma}(a_{i})\|^{2m}=\|\lambda_{\Gamma}\left(a_{i}^{*}a_{i}\right)\|^{m}=\|\lambda_{\Gamma}\left(a_{i}^{*}a_{i}\right)^{m}\|
\]
and similarly, $\|\lambda_{\F}(b_{i})\|^{2m}=\|\lambda_{\F}\left(b_{i}^{*}b_{i}\right)^{m}\|$.
Each $\left(b_{i}^{*}b_{i}\right)^{m}$ is supported on $B_{X}(C(2mR)^{D})$
by (\ref{eq:supports}), hence by Haagerup's inequality (Lemma \ref{lem:Haagerup-iunequality})
we have
\begin{align*}
\|\lambda_{\F}(b_{i})\|^{2m} & =\|\lambda_{\F}\left(b_{i}^{*}b_{i}\right)^{m}\|\\
 & \leq[C(2mR)^{D}]^{\frac{3}{2}}\|\left(b_{i}^{*}b_{i}\right)^{m}\|_{\ell^{2}}\\
 & =[C(2mR)^{D}]^{\frac{3}{2}}\|\left(a_{i}^{*}a_{i}\right)^{m}\|_{\ell^{2}}\\
 & \leq[C(2mR)^{D}]^{\frac{3}{2}}\|\lambda_{\Gamma}(a_{i})\|^{2m}.
\end{align*}
The equality on the third line used that $\phi$ is injective on $B_{Y}(2mR)$,
and the final inequality used (\ref{eq:reverse-inequality}). Hence
\begin{align*}
\|\lambda_{\F}(b_{i})\| & \leq[C(2mR)^{D}]^{\frac{3}{4m}}\|\lambda_{\Gamma}(a_{i})\|\\
 & \leq\left(1+\frac{\epsilon}{3}\right)\|\lambda_{\Gamma}(a_{i})\|\leq\|\lambda_{\Gamma}(a_{i})\|+\frac{\epsilon}{3}
\end{align*}
by our choice of $m$ in (\ref{eq:m-choice}); the last inequality
used that $\|a_{i}\|_{\ell^{1}}=1$ and (\ref{eq:basic-inequality}).
\end{proof}

\section{Proof of Theorem \ref{thm:main-analytical-theorem}}

Here we split into cases when $\Gamma$ is abelian or not. Limit groups
cannot have torsion, so abelian limit groups are of the form $\Z^{r}$
for some $r\in\N$.

\subsection{Proof when $\Gamma=\protect\Z^{r}$}

The case when $\Gamma=\Z^{r}$ must be dealt with by hand here.
\begin{lem}
\label{lem:case-of-Z}Theorem \ref{thm:main-analytical-theorem} holds
when $\Gamma=\Z^{r}$. Moreover, for this sequence of representations
we can ensure 
\begin{equation}
\lim_{i\to\infty}\frac{\Tr(\rho_{i}(z))}{N_{i}}=\tau(z)\label{eq:trace_convergence}
\end{equation}
where $\tau(g)\eqdf\delta_{eg}$ is the delta function at the identity,
extended linearly to a tracial state on $C_{r}^{*}(\Z^{r})$. 
\end{lem}

\begin{proof}
Let $T^{r}\eqdf(S^{1})^{r}$ be the standard $r$-dimensional flat
torus. The Fourier transform gives an isomorphism of $C^{*}$-algebras
\[
\mathcal{F}\colon C_{r}^{*}(\Z^{r})\to C(T^{r}).
\]
For $q\in\N$ let $T_{q}^{r}$ denote the subtorus $(\Z/q\Z)^{r}\subset T^{r}$.
We obtain, via restriction and Fourier transform, a finite dimensional
representation 
\[
C_{r}^{*}(\Z^{r})\xrightarrow{\rho_{q}}C(T_{q}^{r})
\]
that restricts to finite dimensional unitary representation of $\Z^{r}$.
For any $z\in\C[\Z^{r}]$ we have 
\[
\|\rho_{q}(z)\|=\max_{x\in T_{q}^{r}}|\mathcal{F}[z](x)|\to\max_{x\in T^{r}}|\mathcal{F}[z](x)|=\|\lambda_{\Z^{r}}(z)\|
\]
as $q\to\infty$. We have only used here the fact that $T_{q}^{r}$
Hausdorff converges to $T^{r}$ as $q\to\infty$.

We also have

\[
\frac{\Tr(\rho_{q}(z))}{\dim\rho_{q}(z)}=\frac{1}{|T_{q}^{r}|}\sum_{x\in T_{q}^{r}}\mathcal{F}[z](x)\xrightarrow[q\to\infty]{}\int_{T^{r}}\mathcal{F}[z]d\mu=\tau(z)
\]
where $d\mu$ is Lebesgue probability measure on $T^{r}$, and the
convergence is by the definition of the Riemann integral; the last
equality is by Fourier inversion. 
\end{proof}

\subsection{Proof for non-abelian limit groups}

In the following, $\F$ will always denote some (not always the same)
free group, and $\Gamma$ will be a fixed limit group.
\begin{lem}
\label{lem:simple}If $\Gamma$ is a non-abelian limit group, then
the reduced $C^{*}$ algebra of $\Gamma$ is simple (has no non-trivial
closed ideals) and has a unique tracial state.
\end{lem}

\begin{proof}
We claim that any non-abelian FRF group $\Gamma$ has the $P_{\mathrm{nai}}$
property of Bekka, Cowling, and de la Harpe \cite[Def.  4]{BCdlH}.
This states that for any finite set $S\subset\Gamma\backslash\{e\}$,
there is $y\in\Gamma$ of infinite order such that for every $x\in S$,
$x$ and $y$ are free generators of a free rank 2 subgroup of $\Gamma$.

\emph{Proof of Claim.} It is easy to check that since $\Gamma$ is
FRF, two elements $x$ and $y$ are free generators of a free rank
$2$ subgroup of $\Gamma$ if and only if they do not commute. So
to check property $P_{\mathrm{nai}}$ above, it remains to check that
given any finite subset $S\subset\Gamma\backslash\{e\}$, there is
an infinite order $y$ not commuting with any element of $S$.

Because $\Gamma$ is non-abelian, there are two elements $a,b\in\Gamma$
with $[a,b]\neq e$. By the FRF condition, there is a epimorphism
$\phi:\Gamma\to\F$ that is an injection on $S\cup\{e\}\cup\{[a,b]\}$.
In particular, the rank of $\F$ must be at least 2. Since $\phi(S)$
is a finite subset of $\F$ not containing the identity, there is
an (necessarily infinite order) element $f$ not commuting with any
element of $\phi(S)$. Then any preimage of $f,$ say $y$, is infinite
order and does not commute with any element of $S$. \emph{This ends
the proof of the claim.}

The proof of Lemma \ref{lem:simple} now concludes by using \cite[Lemmas 2.1 and 2.2]{BCdlH}.
\end{proof}
\begin{proof}[Proof of Theorem \ref{thm:main-analytical-theorem}]
The upshot of Lemma \ref{lem:simple} is that proving the existence
of a sequence of unitary representations $\{\rho_{i}:\Gamma\to\U(N_{i})\}_{i=1}^{\infty}$
strongly converging to the regular representation reduces to proving
the existence of a sequence with
\begin{equation}
\limsup_{i\to\infty}\|\rho_{i}(z)\|\leq\|\lambda_{\Gamma}(z)\|\label{eq:limsup-aim}
\end{equation}
 for all $z\in\C[\Gamma]$ of unit $\ell^{1}$ norm. We give a proof
of this passage that was also mentioned in the Introduction.

Suppose (\ref{eq:limsup-aim}) holds. Then for any non-principal ultrafilter
$\mathcal{F}$, we form the ultraproduct\footnote{For background on ultrafilters and ultraproducts, see \cite[Appendix A]{BrownOzawa}.}
$C^{*}$-algebra $\mathcal{U}\eqdf\prod_{\mathcal{F}}\rho_{i}(\C[\Gamma])$.
There is a natural $*$-algebra map $\iota:\C[\Gamma]\to\mathcal{U}$.
The inequality (\ref{eq:limsup-aim}) implies 
\begin{equation}
\|\iota(z)\|_{\mathcal{U}}\leq\|\lambda_{\Gamma}(z)\|\label{eq:contraction-temp}
\end{equation}
for all $z\in\C[\Gamma]$. If $\mathcal{U}_{1}$ denotes the closure
of $\iota(\C[\Gamma])$ in $\mathcal{U}$, then inequality (\ref{eq:contraction-temp})
implies that the map $\iota$ extends continuously to $C^{*}$-algebra
map from $C_{r}^{*}(\Gamma)$ to $\mathcal{U}_{1}$. But since we
know $C_{r}^{*}(\Gamma)$ is simple by Lemma \ref{lem:simple}, this
map must be injective. But injective $C^{*}$-algebra maps are isometries
(to their images), so we have for all $z\in\C[\Gamma]$
\[
\|\lambda_{\Gamma}(z)\|=\|\iota(z)\|_{\mathcal{U}}=\lim_{i\to\mathcal{U}}\|\rho_{i}(z)\|.
\]
Since this holds for arbitrary non-principal ultrafilters, it holds
also that $\|\lambda_{\Gamma}(z)\|=\lim_{i\to\infty}\|\rho_{i}(z)\|.$

This reduces our task to proving (\ref{eq:limsup-aim}), which we
begin now. Given $\epsilon>0$ we will prove that there is a unitary
representation $\rho=\rho(U,\epsilon):\Gamma\to\U(N)$ with $N=N(U,\epsilon)$
such that 
\[
\|\rho(z)\|\leq\|\lambda_{\Gamma}(z)\|+\epsilon
\]
for $z\in\C[\Gamma]$ with support in $B\left(\frac{1}{\epsilon}\right)$
and $\|z\|_{\ell^{1}}=1$. By taking $\epsilon\to0$, this will imply
the existence of a sequence $\rho_{i}$ satisfying (\ref{eq:limsup-aim})
for any $z$.

As in the proof of Theorem \ref{thm:Limit-groups-are-C-free} $(\text{\S}$\ref{sec:Proof-of-Theorem-limit-group}),
by taking an $\frac{\epsilon}{3}$-net of the unit $\ell^{1}$ sphere
of the elements in $\C[\Gamma]$ supported on $B\left(\frac{1}{\epsilon}\right)$,
it suffices to prove
\[
\|\rho(a_{i})\|\leq\|\lambda_{\Gamma}(a_{i})\|+\frac{\epsilon}{3}
\]
for a finite collection $\{a_{i}\}_{i\in\mathcal{I}}$ of elements
of $\C[\Gamma]$ with $\|a_{i}\|_{\ell^{1}}=1$.

We apply Theorem \ref{thm:Limit-groups-are-C-free} with $S=B\left(\frac{1}{\epsilon}\right)$
to obtain a homomorphism $\phi:\Gamma\to\F$ with $\F$ free such
that 
\begin{equation}
\|\lambda_{\F}(\phi(a_{i}))\|\leq\|\lambda_{\Gamma}(a_{i})\|+\frac{\epsilon}{6}.\label{eq:free-approximation}
\end{equation}
for all $i\in\I$. Let $b_{i}\eqdf\phi(a_{i})\in\C[\F]$.

The remainder of the proof splits into three cases.

\textbf{A. }If $\F$ is rank 1, i.e. $\F=\Z$ then Lemma \ref{lem:case-of-Z}
tells that there is a finite dimensional unitary representation $\pi$
of $\F$ such that 
\begin{equation}
\|\pi(b_{i})\|\leq\|\lambda_{\F}(b_{i})\|+\frac{\epsilon}{6}\label{eq:free-group-inequality}
\end{equation}
for all $i\in\I$.

\textbf{B. }Otherwise, if one only wants unitary representations in
Theorem \ref{thm:main-analytical-theorem}, then by Haagerup and Thorbjørnsen
\cite[Thm. B]{HaagerupThr} there is a finite dimensional unitary
representation $\pi$ of $\F$ such that (\ref{eq:free-group-inequality})
holds for all $i\in\I$. 

\textbf{C. }If one wants the full strength of Theorem \ref{thm:main-analytical-theorem}
and $\F$ has rank at least $2$, then unitary representations satisfying
(\ref{eq:free-group-inequality}) for all $i\in\I$ exist by the work
of Bordenave and Collins \cite{BordenaveCollins}.

Then let $\rho\eqdf\pi\circ\phi$, a finite dimensional unitary representation
of $\Gamma$. Since $\rho(a_{i})=\pi(b_{i})$, using (\ref{eq:free-approximation})
we obtain
\[
\|\rho(a_{i})\|=\|\pi(b_{i})\|\leq\|\lambda_{\F}(b_{i})\|+\frac{\epsilon}{6}\leq\|\lambda_{\Gamma}(a_{i})\|+\frac{\epsilon}{3}
\]
 for all $i\in\I$ as required.
\end{proof}

\section{Proof of Corollary \ref{cor:matrix-coefficients}\label{sec:Proof-of-matrix-coefficients-version}}

We wish to appeal to results in the literature to deduce Corollary
\ref{cor:matrix-coefficients} from Theorem \ref{thm:main-analytical-theorem}.
To do so, we first establish the following.
\begin{lem}[Strong convergence implies weak convergence]
\label{lem:strong-implies_weak}Let $\Gamma$ be a finitely generated
discrete group such that $C_{r}^{*}(\Gamma)$ has a unique tracial
state. If $\{\rho_{i}:\Gamma\to\U(N_{i})\}_{i=1}^{\infty}$ are a
sequence of finite dimensional unitary representations that strongly
converge to the regular representation of $\Gamma$, then for any
$z\in\C[\Gamma]$
\[
\lim_{i\to\infty}\frac{\Tr(\rho_{i}(z))}{N_{i}}=\tau(z),
\]
where $\tau$ is the unique tracial state on $C_{r}^{*}(\Gamma)$.
$\Tr$ denotes the usual matrix trace on $\U(N_{i})$ extended linearly
to $\C[\U(N_{i})]$.
\end{lem}

We heard this lemma stated by Beno\^{i}t Collins in a talk in Northwestern
University in June 2022. The proof is to our knowledge not in the
literature so we give it here.
\begin{proof}[Proof of Lemma \ref{lem:strong-implies_weak}]
Consider any non-principal ultrafilter $\mathcal{F}$ on $\N$, and
form the ultraproduct $C^{*}$-algebra $\mathcal{U}\eqdf\prod_{\mathcal{F}}\rho_{i}(\C[\Gamma])$.
Let $\mathcal{U}_{1}$ denote the $C^{*}$-subalgebra in $\mathcal{U}$
generated by the images $\hat{\gamma_{i}}$ in $\mathcal{U}$ of the
generators $\gamma_{i}$ of $\Gamma$. Strong convergence implies
that the natural map from $\C[\Gamma]$ to $\mathcal{U}_{1}$ is an
isometric embedding with respect to the norm on $\C[\Gamma]$ coming
from $C_{r}^{*}(\Gamma)$, and hence extends to an isomorphism between
$C_{r}^{*}(\Gamma)$ and $\mathcal{U}_{1}$. On the other hand, 
\[
\lim_{i\to\mathcal{F}}\frac{\Tr\circ\rho_{i}}{N_{i}}
\]
defines a tracial state on $\mathcal{U}_{1}$, and when transferred
to $C_{r}^{*}(\Gamma)$ must coincide with the unique tracial state
there. Since the convergence holds for all non-principal ultrafilters,
the convergence must hold in general.
\end{proof}
\begin{proof}[Proof of Corollary \ref{cor:matrix-coefficients}]
If $\Gamma$ is abelian, let $\rho_{i}$ denote the series of representations
provided by Lemma \ref{lem:case-of-Z}. These strongly converge to
the regular representation and 
\begin{equation}
\lim_{i\to\infty}\frac{\Tr(\rho_{i}(z))}{N_{i}}=\tau(z)\label{eq:weakconvergence-1}
\end{equation}
 where $\tau(z)$ is the canonical faithful tracial state on $C_{r}^{*}(\Gamma)$
(defined by $\tau(e)=1$ and $\tau(g)=0$ for $e\neq g\in\Gamma$). 

If $\Gamma$ is non-abelian, let $\rho_{i}$ denote the strongly convergent
representations provided by Theorem \ref{thm:main-analytical-theorem}.
Since non-abelian limit groups $\Gamma$ have unique tracial states
on $C_{r}^{*}(\Gamma)$ by Lemma \ref{lem:simple}, Lemma \ref{lem:strong-implies_weak}
implies that we have (\ref{eq:weakconvergence-1}) also in this case.

Note that for unitary matrices, $u_{1},\ldots,u_{r}$, any non-commutative
polynomial (possibly with matrix coefficients) in the $u_{i}$ is
another polynomial (possibly with matrix coefficients) in the Hermitian
matrices $u_{i}+u_{i}^{*}$ and $i(u_{i}-u_{i}^{*})$, and vice versa.

Given this observation, (\ref{eq:weakconvergence-1}) together with
strong convergence of the sequence $\rho_{i}$ and faithfulness of
the trace $\tau$ used as inputs to \cite[Prop. 7.3]{Male} yield
Corollary \ref{cor:matrix-coefficients}.
\end{proof}

\section{Proof of Proposition \ref{prop:Klein-bottle}}

If $\Gamma\leq\Lambda$ are countable groups and $\rho:\Gamma\to\U(H)$
is a unitary representation of $\Gamma$ on a separable Hilbert space
$H$, the induced representation
\[
\Ind_{\Gamma}^{\Lambda}\rho\eqdf\C[\Lambda]\otimes_{\Gamma}H
\]
 has an invariant hermitian inner product for which $g_{i}\otimes_{\Gamma}e_{j}$
is an orthonormal basis, where $g_{1},\ldots,g_{K},\ldots$ denote
left coset representatives for $\Gamma$ in $\Lambda$ and $\{e_{j}\}_{j=1}^{\dim(H)}$
are an orthonormal basis for $H$.

The proof of both cases of Proposition \ref{prop:Klein-bottle} rely
on the following lemma, which may be of independent interest. 
\begin{lem}
\label{lem:induction}Let $\Lambda$ be any discrete group, $\Gamma$
a finite index subgroup, and $\rho_{i}:\Gamma\to\U(N_{i})$ finite
dimensional unitary representations for which the conclusion of Corollary
\ref{cor:matrix-coefficients} holds. Then the induced unitary representations
$\Ind_{\Gamma}^{\Lambda}\rho_{i}$ strongly converge to the regular
representation of $\Lambda$.
\end{lem}

\begin{proof}
Let $\rho:\Gamma\to\U(H)$ be any unitary representation of $\Gamma$
on a Hilbert space $H$. Let $g_{1},\ldots,g_{K}$ denote left coset
representatives for $\Gamma$ in $\Lambda$. For each $1\leq k\leq K$
and $h\in\Lambda$ we have 
\[
h[g_{k}\otimes_{\Gamma}v]=g_{\kappa(k,h)}\gamma(k,h)\otimes_{\Gamma}e_{\ell}=g_{\kappa(k,h)}\otimes_{\Gamma}\rho(\gamma(k,h))v
\]
This map is isometrically conjugate to the map 
\[
\sum_{k}E_{k,\kappa(k,h)}\otimes_{\C}\rho(\gamma(k,h))\in\End(\ell^{2}(\Lambda/\Gamma))\otimes_{\C}\End(H).
\]
where $E_{p,q}(g_{i})\eqdf\delta_{ip}g_{q}$ are the elementary matrices.
The isometric conjugacy does not depend on $\rho$ or $h$, only $\Lambda$
and $\Gamma$ and the choice of coset representatives. This means
for any $z\in\C[\Gamma]$ the map $[\Ind_{\Gamma}^{\Lambda}\rho](z)$
is conjugate to some 
\begin{equation}
\sum_{g\in\Lambda}a(g)\otimes\rho(g)\in\End(\ell^{2}(\Lambda/\Gamma))\otimes_{\C}\End(H)\label{eq:tensor-form-induced}
\end{equation}
where $g\mapsto a(g)$ is finitely supported and the coefficients
$a(g)$ do not depend on $\rho$; i.e. as $\rho$ varies they may
be taken the same for a fixed $z$.

Applying the conclusion of of Corollary \ref{cor:matrix-coefficients}
to (\ref{eq:tensor-form-induced}) for a sequence of $\rho_{i}:\Gamma\to\U(N_{i})$
we learn that 
\[
\lim_{i\to\infty}\|[\Ind_{\Gamma}^{\Lambda}\rho_{i}](z)\|=\|\sum_{g\in\Lambda}a(g)\otimes\lambda_{\Gamma}(g)\|.
\]
 But now applying our previous argument in reverse, $\sum_{g\in\Lambda}a(g)\otimes\lambda_{\Gamma}(g)$
is isometrically conjugate to $[\Ind_{\Gamma}^{\Lambda}\lambda_{\Gamma}](z)=\lambda_{\Lambda}(z)$,
since $\C[\Lambda]\otimes_{\Gamma}\ell^{2}[\Gamma]\cong\ell^{2}(\Lambda)$
as a left $\Lambda$ module. Hence the conclusion
\[
\lim_{i\to\infty}\|[\Ind_{\Gamma}^{\Lambda}\rho_{i}](z)\|=\|\lambda_{\Lambda}(z)\|.
\]
\end{proof}
\begin{proof}[Proof of Proposition \ref{prop:Klein-bottle}]
The Klein bottle has an orientable double cover with Euler characteristic
$0=2\times0$, hence a torus. The non-orientable surface $(\R P^{2})^{\#3}$
with Euler characteristic $-1$ has a orientable double cover with
Euler characteristic $-2$, hence it is a genus $2$ orientable surface.
This means that their fundamental groups have index 2 subgroups respectively
isomorphic to $\Z^{2}$ and $\Gamma_{2}$, the fundamental group of
a orientable genus 2 surface. Since $\Z^{2}$ and $\Gamma_{2}$ are
both limit groups, Corollary \ref{cor:matrix-coefficients} applies
to both of them. Therefore Lemma \ref{lem:induction} implies that
the fundamental groups of both the Klein bottle and $(\R P^{2})^{\#3}$
have finite dimensional unitary representations that strongly converge
to their regular representations.
\end{proof}

\appendix

\section{\label{sec:Spectral-gaps-of}Spectral gaps of hyperbolic surfaces}

The purpose of this appendix is to explain how the following theorem
can be deduced from Corollary \ref{cor:matrix-coefficients}.
\begin{thm}
\label{thm:Main Theorem}Let $X$ be a compact hyperbolic surface.
There exists a sequence of Riemannian covers $\left\{ X_{i}\right\} _{i\in\mathbb{N}}$
of $X$ with genera $g(i)\to\infty$ as $i\to\infty$ such that for
any $\epsilon>0$, for $i$ large enough depending on $\epsilon$,
\[
\spec\left(\Delta_{X_{i}}\right)\cap\left[0,\frac{1}{4}-\epsilon\right)=\spec\left(\Delta_{X}\right)\cap\left[0,\frac{1}{4}-\epsilon\right),
\]
where the multiplicities are the same on either side. 
\end{thm}

To see that we can take all surfaces to be arithmetic we use the following
argument. Let 
\[
\Gamma_{0}(15)\eqdf\left\{ \left(\begin{array}{cc}
a & b\\
c & d
\end{array}\right)\in\SL_{2}(\Z)\,:\,c\equiv0\bmod15\right\} .
\]
The cusped hyperbolic surface $Y_{0}(15)\eqdf\Gamma_{0}(15)\backslash\mathbb{H}$
has no spectrum in $(0,\frac{1}{4})$ by a result of Huxley \cite[Thm., pg. 250]{Huxley}.
Let $D_{3,5}$ denote the quaternion algebra over $\mathbf{Q}$ generated
by $i,j,k$ such that 
\[
i^{2}=3,\,j^{2}=5,\,ij=-ji=k.
\]
Then $D_{3,5}$ is a division algebra with discriminant 15 \cite[Ex. 8.27]{BergeronBook}.
Let $\mathcal{O}$ denote a maximal order\footnote{See \cite[Ex. 8.27]{BergeronBook} for an explicit maximal order.}
in $D_{3,5}$ and $\mathcal{O}^{1}$ the elements of norm 1 in $\text{\ensuremath{\mathcal{O}}}$.
Then $\text{\ensuremath{\mathcal{O}^{1}}}$ embeds as a cocompact
subgroup of $\mathrm{PSL}_{2}(\mathbf{R})$; let $X=\ensuremath{\mathcal{O}^{1}}\backslash\mathbb{H}$.
By the work of Jacquet and Langlands \cite{JacquetLanglands} (see
\cite[Thm. 8.18]{BergeronBook} for a convenient concise reference)
every eigenvalue of $X$ is an eigenvalue of $Y_{0}(15)$ and hence
$X$ has no eigenvalues in $(0,\frac{1}{4})$. 

Taking this $X$ in Theorem \ref{thm:Main Theorem}, one obtains a
different proof of \cite[Corollary 1.3]{HM21} with a slightly stronger
conclusion, i.e. there exists a sequence of compact \emph{arithmetic
}hyperbolic surfaces $\left\{ X_{i}\right\} _{i\in\mathbb{N}}$ with
genera $g\left(X_{i}\right)\to\infty$ and $\lambda_{1}\left(X_{i}\right)\to\frac{1}{4}$.
Such a sequence of covering surfaces also have systoles uniformly
bounded away from $0$, also in contrast to the proof of \cite[Corollary 1.3]{HM21}
(this conclusion on the systole is independent of arithmeticity).

\subsection{Set up}

For any $n\in\mathbb{N}$, let $[n]\eqdf\{1,\ldots,n\}$ and $S_{n}$
denote the group of permutations of $[n]$. Let $X$ be a fixed compact
hyperbolic surface with genus $g\geqslant2$. We view $X$ as 
\[
X=\Gamma\backslash\mathbb{H},
\]
where $\Gamma$ is a discrete, torsion free subgroup of $\text{PSL}_{2}\left(\mathbb{R}\right)$,
isomorphic to the surface group $\Lambda_{g}$. Given any $\phi\in\text{Hom}\left(\Gamma,S_{n}\right)$
we define an action of $\Gamma$ on $\mathbb{H}\times[n]$ by 
\[
\gamma\left(z,x\right)\eqdf\left(\gamma z,\phi(\gamma)[x]\right).
\]
Then we obtain a degree $n$ covering space $X_{\phi}$ of $X$ by
\begin{equation}
X_{\phi}\eqdf\Gamma\backslash_{\phi}\left(\mathbb{H}\times[n]\right).\label{eq:covering surface def}
\end{equation}
Let $V_{n}\eqdf\ell^{2}\left([n]\right)$ and $V_{n}^{0}\subset V_{n}$
the subspace of functions with zero mean. Then $S_{n}$ acts on $V_{n}$
via $\text{std}$, the standard representation by $0$-$1$ matrices,
and $V_{n}^{0}$ is the $n-1$ dimensional irreducible component.
Throughout this appendix, we let $\left\{ \rho_{i}\right\} _{i\in\mathbb{N}}$
be a sequence of $N_{i}$-dimensional unitary representations of $\Gamma$
that factor through $S_{N_{i}}$ by
\begin{equation}
\Gamma\xrightarrow{\phi_{i}}S_{N_{i}}\xrightarrow{\text{std}}\text{End}\left(V_{N_{i}}^{0}\right),\label{eq:Factor-through-Sn}
\end{equation}
such that for any $r\in\N$ and finitely supported map $a:\Gamma\to\mathrm{Mat}_{r\times r}(\C)$,
we have 
\begin{equation}
\limsup_{i\to\infty}\|\sum_{\gamma\in\Gamma}a(\gamma)\otimes\rho_{i}(\gamma)\|\leq\|\sum_{\gamma\in\Gamma}a(\gamma)\otimes\lambda(\gamma)\|,\label{eq:Appendix-limsup-bound}
\end{equation}
as provided by Corollary \ref{cor:matrix-coefficients}. Note that
by approximation by finite-rank operators on either side (as in \cite[Proof of Prop. 6.3]{HM21})
the property in (\ref{eq:Appendix-limsup-bound}) extends easily to
the case of 
\[
a:\Gamma\to\mathcal{K}
\]
where $\K$ are the compact operators on a separable Hilbert space.
We use this extension in the sequel.

Then through $\left\{ \rho_{i}\right\} _{i\in\mathbb{N}}$, we obtain
a sequence of degree-$N_{i}$ covering surfaces $\left\{ X_{i}\right\} _{i\in\mathbb{N}}$
from (\ref{eq:covering surface def}).

\subsection{Function spaces}

For the convenience of the reader we recall the following function
spaces from \cite[Section 2.2]{HM21}. We define $L_{\new}^{2}\left(X_{i}\right)$
to be the space of $L^{2}$ functions on $X_{i}$ orthogonal to all
lifts of $L^{2}$ functions from $X$. Then
\[
L^{2}\left(X_{i}\right)\cong L_{\text{\ensuremath{\new}}}^{2}\left(X\right)\oplus L^{2}\left(X\right).
\]
We fix $F$ to be a Dirichlet fundamental domain for $X$. Let $C^{\infty}\left(\mathbb{H};V_{N_{i}}^{0}\right)$
denote the smooth $V_{N_{i}}^{0}$-valued functions on $\mathbb{H}$.
There is an isometric linear isomorphism between 
\[
C^{\infty}\left(X_{i}\right)\cap L_{\text{\ensuremath{\new}}}^{2}\left(X_{i}\right),
\]
and the space of smooth $V_{N_{i}}^{0}$-valued functions on $\mathbb{H}$
satisfying 
\[
f\left(\gamma z\right)=\rho_{i}\left(\gamma\right)f\left(z\right),
\]
 for all $\gamma\in\Gamma$, with finite norm
\[
\|f\|_{L^{2}(F)}^{2}\eqdf\int_{F}\|f(z)\|_{V_{N_{i}}^{0}}^{2}d\mu_{\mathbb{H}}\left(z\right)<\infty.
\]
We denote the space of such functions by $C_{\phi_{i}}^{\infty}\left(\mathbb{H};V_{N_{i}}^{0}\right).$
The completion of $C_{\phi_{i}}^{\infty}\left(\H;V_{N_{i}}^{0}\right)$
with respect to $\|\bullet\|_{L^{2}(F)}$ is denoted by $L_{\phi_{i}}^{2}\left(\H;V_{N_{i}}^{0}\right)$;
the isomorphism above extends to one between $L_{\text{\ensuremath{\new}}}^{2}\left(X_{i}\right)$
and $L_{\phi_{i}}^{2}\left(\mathbb{H};V_{N_{i}}^{0}\right)$. 

We introduce the following Sobolev spaces. Let $H^{2}\left(\H\right)$
denote the completion of $C_{c}^{\infty}\left(\H\right)$ with respect
to the norm
\[
\|f\|_{H^{2}\left(\H\right)}^{2}\eqdf\|f\|_{L^{2}\left(\H\right)}^{2}+\|\Delta f\|_{L^{2}\left(\H\right)}^{2}.
\]
Let $C_{c,\phi_{i}}^{\infty}\left(\H;V_{N_{i}}^{0}\right)$ denote
the subset of $C_{\phi_{i}}^{\infty}\left(\H;V_{N_{i}}^{0}\right)$
consisting of functions which are compactly supported modulo $\Gamma$.
We let $H_{\phi_{i}}^{2}\left(\H;V_{N_{i}}^{0}\right)$ denote the
completion of $C_{c,\phi_{i}}^{\infty}\left(\H;V_{N_{i}}^{0}\right)$
with respect to the norm
\[
\|f\|_{H_{\phi_{i}}^{2}\left(\H;V_{N_{i}}^{0}\right)}^{2}\eqdf\|f\|_{L^{2}(F)}^{2}+\|\Delta f\|_{L^{2}(F)}^{2}.
\]
We let $H^{2}\left(X_{i}\right)$ denote the completion of $C_{c}^{\infty}\left(X_{i}\right)$
with respect to the norm 
\[
\|f\|_{H^{2}(X_{i})}^{2}\eqdf\|f\|_{L^{2}(X_{i})}^{2}+\|\Delta f\|_{L^{2}(X_{i})}^{2}.
\]
Viewing $H^{2}\left(X_{i}\right)$ as a subspace of $L^{2}\left(X_{\phi_{i}}\right)$,
we let 
\[
H_{\new}^{2}\left(X_{i}\right)\eqdf H^{2}\left(X_{i}\right)\cap L_{\new}^{2}\left(X_{i}\right).
\]
There is an isometric isomorphism between $H_{\new}^{2}\left(X_{i}\right)$
and $H_{\phi_{i}}^{2}\left(\H;V_{N_{i}}^{0}\right)$ that intertwines
the two relevant Laplacian operators.

\subsection{Operators on $\mathbb{H}$}

For $s\in\mathbb{C}$ with $\text{Re}(s)>\frac{1}{2}$, let
\begin{align*}
R_{\mathbb{H}}(s) & :L^{2}\left(\mathbb{H}\right)\to L^{2}\mathbb{\left(H\right)},\\
R_{\mathbb{H}}(s) & \eqdf\left(\Delta_{\mathbb{H}}-s(1-s)\right)^{-1},
\end{align*}
be the resolvent on the upper half plane. Then $R_{\mathbb{H}}(s)$
is an integral operator with radial kernel $R_{\mathbb{H}}(s;r)$.
Let $\chi_{0}:\mathbb{R}\to\left[0,1\right]$ be a smooth function
such that 
\[
\chi_{0}\left(t\right)=\begin{cases}
0 & \text{if \ensuremath{t\leqslant0}, }\\
1 & \text{if \ensuremath{t\geqslant1}. }
\end{cases}.
\]
For $T>0$, we define a smooth cutoff function $\chi_{T}$ by $\chi_{T}(t)\eqdf\chi_{0}(t-T).$
We then define the operator $R_{\mathbb{H}}^{(T)}(s):L^{2}\left(\mathbb{H}\right)\to L^{2}\left(\mathbb{H}\right)$
to be the integral operator with radial kernel
\[
R_{\mathbb{H}}^{(T)}(s;r)\eqdf\chi_{T}(r)R_{\mathbb{H}}(s;r).
\]
Following \cite[Section 5.2]{HM21} we define $L_{\mathbb{H}}^{(T)}(s):L^{2}\left(\mathbb{H}\right)\to L^{2}\left(\mathbb{H}\right)$
to be the integral operator with radial kernel
\[
\mathbb{L}^{(T)}(s;r)\eqdf\left(-\frac{\partial^{2}}{\partial r^{2}}\left[\chi_{T}\right]-\frac{1}{\tanh r}\frac{\partial}{\partial r}\left[\chi_{T}\right]\right)R_{\mathbb{H}}(s;r)-2\frac{\partial}{\partial r}\left[\chi_{T}\right]\frac{\partial R_{\mathbb{H}}}{\partial r}(s;r).
\]
It is proved in \cite[Lemma 5.3]{HM21} that for any $f\in C_{c}^{\infty}\left(\mathbb{H}\right)$and
$s\in\left[\frac{1}{2},1\right]$, we have 
\begin{enumerate}
\item $R_{\mathbb{H}}^{(T)}(s)f\in H^{2}\left(\mathbb{H}\right).$
\item $\left(\Delta-s(1-s)\right)R_{\mathbb{H}}^{(T)}(s)f=f+\mathbb{L}_{\mathbb{H}}^{(T)}(s)f$
as equivalence classes of $L^{2}$ functions.
\end{enumerate}
It is also proved, as a consequence of \cite[Lemma 5.2]{HM21}, that
for any $s_{0}>\frac{1}{2}$ we can choose a $T=T(s_{0})$ such that
for all $s\in[s_{0},1]$ we have 
\begin{equation}
\|\mathbb{L}_{\mathbb{H}}^{(T)}(s)\|_{L^{2}}\leq\frac{1}{4}.\label{eq:infinity operator bound}
\end{equation}

\subsection{Proof of Theorem \ref{thm:Main Theorem} }

Recall that $\left\{ \rho_{i}\right\} _{i\in\mathbb{N}}$ is a sequence
of strongly convergent representations of the form (\ref{eq:Factor-through-Sn})
that satisfy (\ref{eq:Appendix-limsup-bound}) as guaranteed by Corollary
\ref{cor:matrix-coefficients}. As in \cite[Section 5.3]{HM21}, we
define
\begin{align*}
R_{\mathbb{H},i}^{(T)}(s;x,y) & \eqdf R_{\mathbb{H}}^{(T)}(s;x,y)\mathrm{Id}_{V_{N_{i}}^{0}},\\
\mathbb{L}_{\mathbb{H},i}^{(T)}(s;x,y) & \eqdf\mathbb{L}_{\mathbb{H}}^{(T)}(s;x,y)\mathrm{Id}_{V_{N_{i}}^{0}}.
\end{align*}
We define $R_{\mathbb{H},i}^{(T)}\left(s\right)$, $\mathbb{L}_{\mathbb{H},i}^{(T)}\left(s\right)$
to be the corresponding integral operators. We have the following
analogue of \cite[Lemma 5.5]{HM21}.
\begin{lem}
\label{lem:bounded op}For all $s\in\left[\frac{1}{2},1\right]$, 
\begin{enumerate}
\item The integral operator $R_{\mathbb{H},i}^{(T)}\left(s\right)$ is well-defined
on $C_{c,\phi_{i}}^{\infty}\left(\mathbb{H};V_{N_{i}}^{0}\right)$
and extends to a bounded operator 
\[
R_{\mathbb{H},i}^{(T)}\left(s\right):L_{\phi_{i}}^{2}\left(\mathbb{H};V_{N_{i}}^{0}\right)\to H_{\phi_{i}}^{2}\left(\mathbb{H};V_{N_{i}}^{0}\right).
\]
\item The integral operator $\mathbb{L}_{\mathbb{H},i}^{(T)}\left(s\right)$
is well-defined on $C_{c,\phi_{i}}^{\infty}\left(\mathbb{H};V_{N_{i}}^{0}\right)$
and and extends to a bounded operator on $L_{\phi}^{2}\left(\mathbb{H};V_{N_{i}}^{0}\right)$.
\item We have 
\[
\left[\Delta-s(1-s)\right]R_{\mathbb{H},i}^{(T)}\left(s\right)=1+\mathbb{L}_{\mathbb{H},i}^{(T)}\left(s\right)
\]
 as an identity of operators on $L_{\phi_{i}}^{2}\left(\mathbb{H};V_{N_{i}}^{0}\right)$.
\end{enumerate}
\end{lem}

The proof of Lemma \ref{lem:bounded op} easily follows from the proof
of \cite[Lemma 5.5]{HM21}, simplified in places by the compactness
of the fundamental domain $F$ in the current setting. 

We have an isomorphism of Hilbert spaces 
\[
L_{\phi_{i}}^{2}\left(\mathbb{H};V_{N_{i}}^{0}\right)\cong L^{2}\left(F\right)\otimes V_{N_{i}}^{0},
\]
given by 
\[
f\mapsto\sum_{e_{i}}\langle f\lvert_{F},e_{i}\rangle\otimes e_{i},
\]
where $\{e_{j}\}_{j=1}^{N_{i}-1}$ is some choice of basis for $V_{N_{i}}^{0}$.
After conjugation by this isomorphism, the operator $\mathbb{L}_{\mathbb{H},i}^{(T)}\left(s\right)$
becomes
\begin{equation}
\mathbb{L}_{\mathbb{H},i}^{(T)}\left(s\right)\cong\sum_{\gamma\in S}a_{\gamma}^{(T)}\left(s\right)\otimes\rho_{i}\left(\gamma^{-1}\right),\label{eq:L-N-conjugation}
\end{equation}
where
\begin{align*}
a_{\gamma}^{(T)}\left(s\right) & :L^{2}\left(F\right)\to L^{2}\left(F\right),\\
a_{\gamma}^{(T)}\left(s\right)\left[f\right]\left(x\right) & \eqdf\int_{y\in F}\mathbb{L}_{\mathbb{H}}^{(T)}\left(s;\gamma x,y\right)f\left(y\right)d\mu_{\mathbb{H}}\left(y\right).
\end{align*}
Since $\mathbb{L}_{\mathbb{H}}^{(T)}\left(s,\gamma x,y\right)$ is
only non-zero when $d\left(\gamma x,y\right)\leqslant T+1$, in (\ref{eq:L-N-conjugation})
one can take $S=S\left(T\right)\subset\Gamma$ to be finite. Since
$\mathbb{L}_{\mathbb{H}}^{(T)}\left(s;\gamma x,y\right)$ is smooth
and bounded it follows that the operators $a_{\gamma}^{(T)}\left(s\right)$
are Hilbert-Schmidt and therefore compact. We define 
\[
\mathcal{L}_{s,\infty}^{(T)}\eqdf\sum_{\gamma\in S}a_{\gamma}^{(T)}(s)\otimes\lambda\left(\gamma^{-1}\right).
\]
Under the isomorphism
\begin{align*}
L^{2}\left(F\right)\otimes\ell^{2}\left(\Gamma\right) & \cong L^{2}\left(\mathbb{H}\right),\\
f\otimes\delta_{\gamma} & \mapsto f\circ\gamma^{-1},
\end{align*}
(with $f\circ\gamma^{-1}$ extended by zero from a function on $\gamma F$)
the operator $\mathcal{L}_{s,\infty}^{(T)}$ is conjugated to 
\[
\mathbb{L}_{\mathbb{H}}^{(T)}(s):L^{2}\left(\mathbb{H}\right)\to L^{2}\left(\mathbb{H}\right).
\]

To prove Theorem \ref{thm:Main Theorem}, we need to replace the probabilistic
bound \cite[Lemma 6.3]{HM21} by a deterministic one. 
\begin{prop}
\label{prop:Operator norm bound}For any $s_{0}>\frac{1}{2}$ there
is a $T=T(s_{0})>0$ such that for any fixed $s\in[s_{0},1]$ there
is an $I\left(s_{0},s\right)$ with
\[
\|\mathcal{L}_{s,\phi_{i}}^{(T)}\|_{L^{2}(F)\otimes V_{N_{i}}^{0}}\leq\frac{1}{4},
\]
for all $i\geqslant I$.
\end{prop}

\begin{proof}
Let $s_{0}>\frac{1}{2}$ and a fixed $s\in[s_{0},1]$ be given. By
(\ref{eq:infinity operator bound}) we can find a $T\left(s_{0}\right)$
such that
\begin{equation}
\|\mathcal{L}_{s,\infty}^{(T)}\|_{L^{2}(F)\otimes\ell^{2}(\Gamma)}\leq\frac{1}{8}.\label{eq:temp-l;infty-bound}
\end{equation}
Recall that the coefficients $a_{\gamma}(s)$ are supported on a finite
set $S=S(T)\subset\Gamma$. Because the $a_{\gamma}(s)$ are compact,
we apply (\ref{eq:Appendix-limsup-bound}) (and the following remark)
to the operators $\mathbb{L}_{\mathbb{H},i}^{(T)}\left(s\right)$
to find that there is $I\in\N$ such that for all $i\geq I$
\[
\|\mathcal{L}_{s,\phi_{i}}^{(T)}\|_{L^{2}(F)\otimes V_{N_{i}}^{0}}\leq\|\mathcal{L}_{s,\infty}^{(T)}\|_{L^{2}(F)\otimes\ell^{2}(\Gamma)}+\frac{1}{8}\leq\frac{1}{4}.
\]
\end{proof}
We can now prove Theorem \ref{thm:Main Theorem}.
\begin{proof}[Proof of Theorem \ref{thm:Main Theorem}]
Given $\epsilon>0$ let $s_{0}=\frac{1}{2}+\sqrt{\epsilon}$ so that
$s_{0}\left(1-s_{0}\right)=\frac{1}{4}-\epsilon$. Let $T=T\left(s_{0}\right)$
be the value provided by Proposition \ref{prop:Operator norm bound}
for this $s_{0}$. 

We use a finite net to control all values of $s\in\left[s_{0},1\right]$.
Using \cite[Lemma 6.1]{HM21} as in \cite[Proof of Thm. 1.1]{HM21}
tells us that there is a finite set $Y=Y(s_{0})$ of points in $[s_{0},1]$
such that for any $s\in[s_{0},1]$ there is $s'\in Y$ with
\begin{equation}
\|\mathcal{L}_{s,\phi_{i}}^{(T)}-\mathcal{L}_{s',\phi_{i}}^{(T)}\|\leq\frac{1}{4}\label{eq:L-deviations}
\end{equation}
 for all $i$. 

Combining (\ref{eq:L-deviations}) with Proposition \ref{prop:Operator norm bound}
applied to $\mathcal{L}_{s,\phi_{i}}^{(T)}$ for every $s\in Y$ we
find that there is an $I\left(s_{0}\right)$ such that for all $s\in[s_{0},1]$
and $i\geq I(s_{0})$
\begin{equation}
\|\mathbb{L}_{\mathbb{H},i}^{(T)}\left(s\right)\|_{L_{\new}^{2}(X_{i})}=\|\mathcal{L}_{s,\phi_{i}}^{(T)}\|_{L^{2}(F)\otimes V_{N_{i}}^{0}}\leq\frac{1}{2}.\label{eq:Surface-op-bound}
\end{equation}

By Lemma \ref{lem:bounded op}, for $s>\frac{1}{2}$ $R_{\mathbb{H},i}^{(T)}\left(s\right)$
is a bounded operator from $L_{\text{\ensuremath{\new}}}^{2}\left(X_{i}\right)$
to $H_{\new}^{2}\left(X_{i}\right)$. By Lemma \ref{lem:bounded op}
we have that
\[
\left(\Delta_{X_{\phi}}-s(1-s)\right)R_{\mathbb{H},i}^{(T)}\left(s\right)=1+\mathbb{L}_{\mathbb{H},i}^{(T)}\left(s\right),
\]
on $L_{\text{new}}^{2}\left(X_{i}\right)$. From (\ref{eq:Surface-op-bound})
for all $i\geqslant I\left(s_{0}\right)$ and $s\in\left[s_{0},1\right]$
$\left(1+\mathbb{L}_{\mathbb{H},i}^{(T)}\left(s\right)\right)^{-1}$
exists as a bounded operator on $L_{\new}^{2}\left(X_{i}\right)$.
We now get that for all $i\geqslant I\left(s_{0}\right)$ and all
$s\in\left[s_{0},1\right],$
\[
\left(\Delta_{X_{i}}-s\left(1-s\right)\right)R_{\mathbb{H},i}^{(T)}\left(s\right)\left(1+\mathbb{L}_{\mathbb{H},i}^{(T)}\left(s\right)\right)^{-1}=1,
\]
and we conclude that $\left(\Delta_{X_{i}}-s\left(1-s\right)\right)$
has a bounded right inverse from $L_{\text{\ensuremath{\new}}}^{2}\left(X_{i}\right)$
to $H_{\new}^{2}\left(X_{i}\right)$, implying that for $i\geqslant I\left(s_{0}\right)$,
$\Delta_{X_{i}}$ has no new eigenvalues $\lambda$ with $\lambda\leq s_{0}\left(1-s_{0}\right)=\frac{1}{4}-\epsilon$. 
\end{proof}
\bibliographystyle{amsalpha}
\bibliography{strong_convergence}

\noindent Will Hide, \\
Department of Mathematical Sciences,\\
Durham University, \\
Lower Mountjoy, DH1 3LE Durham,\\
United Kingdom

\noindent \texttt{william.hide@durham.ac.uk}~\\
\texttt{}~\\
Larsen Louder, \\
Department of Mathematics,\\
University College London,\\
Gower Street, London WC1E 6BT,\\
United Kingdom\\
\texttt{l.louder@ucl.ac.uk}~\\
\texttt{}~\\
Michael Magee, \\
Department of Mathematical Sciences,\\
Durham University, \\
Lower Mountjoy, DH1 3LE Durham,\\
United Kingdom

\noindent \texttt{michael.r.magee@durham.ac.uk}\\

\end{document}